\documentclass[reqno]{amsproc}

\usepackage[margin=1in,marginparwidth=2cm, marginparsep=0.2cm]{geometry}
\usepackage{setspace, fullpage}
\usepackage{amsmath,amsthm,amscd,amssymb,bbm,mathrsfs,latexsym}
\usepackage[colorlinks,citecolor=blue,linkcolor=blue,pagebackref,hypertexnames=false]{hyperref}

\newcommand{\vep}{\varepsilon}
\usepackage{mathtools}

\usepackage{tikz}
\usetikzlibrary{arrows,shapes}
\usetikzlibrary{decorations.markings}

\tikzstyle{white}=[circle,draw=black!100,fill=white!100,thick,inner sep=0pt,minimum size =2mm]
\tikzstyle{black}=[circle,draw=black!100,fill=black!100,thick,inner sep=0pt,minimum size =2mm]
\tikzstyle{graywhite}=[circle,draw=gray!100,fill=white!100,thick,inner sep=0pt,minimum size =2mm]

\tikzset{->-/.style={decoration={
  markings,
  mark=at position .5 with {\arrow[scale=1.3]{>}}},
  postaction={decorate}}
  }

\numberwithin{equation}{section}
\theoremstyle{plain}
\newtheorem{theorem}{{\bf Theorem}}[section]
\newtheorem{lemma}[theorem]{{\bf Lemma}}
\newtheorem{corollary}[theorem]{Corollary}

\theoremstyle{definition}
\newtheorem{definition}[theorem]{{\bf Definition}}

\theoremstyle{remark}

\numberwithin{equation}{section}

\begin{document}

\title{Comparing the generalized roundness of metric spaces}

\author{Lukiel Levy-Moore}
\address{Department of Economics, New York University, New York, NY 10012, USA}
\email{LukielLevyMoore@nyu.edu}

\author{Margaret Nichols}
\address{Department of Mathematics, University of Chicago, Chicago, IL 60637, USA}
\email{mnichols@math.uchicago.edu}

\author{Anthony Weston}
\address{Department of Mathematics and Statistics, Canisius College, Buffalo, NY 14208, USA}
\email{westona@canisius.edu}
\address{Department of Decision Sciences, University of South Africa, UNISA 0003, South Africa}
\email{westoar@unisa.ac.za}



\begin{abstract}
Motivated by the local theory of Banach spaces we introduce a notion of finite representability
for metric spaces. This allows us to develop a new technique for comparing the generalized
roundness of metric spaces. We illustrate this technique in two different ways by applying it
to Banach spaces and metric trees. In the realm of Banach spaces we obtain results such as
the following: (1) if $\mathcal{U}$ is any ultrafilter and $X$ is any Banach space, then the
second dual $X^{\ast\ast}$ and the ultrapower $(X)_{\mathcal{U}}$ have the same generalized
roundness as $X$, and (2) no Banach space of positive generalized roundness is uniformly
homeomorphic to $c_{0}$ or $\ell_{p}$, $2 < p < \infty$. Our technique also leads to the
identification of new classes of metric trees of generalized roundness one. In particular,
we give the first examples of metric trees of generalized roundness one that have finite diameter.
These results on metric trees provide a natural sequel to a paper of Caffarelli \textit{et al}.\
\cite{Dcw}. In addition, we show that metric trees of generalized roundness one possess
special Euclidean embedding properties that distinguish them from all other metric trees.
\end{abstract}

\maketitle

\section{Introduction}\label{sec:1}

Direct calculation of the generalized roundness of an infinite metric space is, in general,
a difficult task. In this paper we develop a versatile technique for comparing the generalized
roundness of metric spaces. This leads to substantial new insights into the generalized roundness
of Banach spaces and metric trees.

\begin{definition}\label{def:0}
The \textit{generalized roundness} of a metric space $(X,d)$, denoted by $\wp_{(X,d)}$ or simply $\wp_{X}$,
is the supremum of the set of all $p \geq 0$ that satisfy the following condition: for all integers
$k \geq 2$ and all choices of (not necessarily distinct) points $a_{1}, \ldots, a_{k}, b_{1}, \ldots, b_{k} \in X$,
we have
\begin{eqnarray}\label{ONE}
\sum\limits_{1 \leq i < j \leq k} \left\{ d(a_{i},a_{j})^{p} + d(b_{i},b_{j})^{p} \right\}
& \leq & \sum\limits_{1 \leq i,j \leq k} d(a_{i},b_{j})^{p}.
\end{eqnarray}
The configuration of points $D_{k} = [a_{1}, \ldots, a_{k}; b_{1}, \ldots, b_{k}] \subseteq X$ underlying (\ref{ONE})
will be called a \textit{simplex} in $X$. We will say that $p \geq 0$ is a \textit{generalized roundness exponent}
for $(X,d)$ if (\ref{ONE}) holds for every simplex in $X$.
\end{definition}

The key idea of this paper is to take an indirect approach to the
calculation of $\wp_{(X,d)}$ that is especially well-suited to the analysis of infinite metric spaces.

The notion of generalized roundness was introduced by Enflo \cite{En1} to study universal
uniform embedding spaces. By showing that such spaces must have generalized roundness zero, Enflo was
able to prove that Hilbert spaces are not universal uniform embedding spaces. This resolved a
prominent question of Smirnov. Sometime later, Lennard \textit{et al}.\ \cite{Ltw} exhibited
an important connection between generalized roundness and the classical isometric embedding notion
of negative type. Lafont and Prassidis \cite{Laf} used this connection to show that if a
finitely generated group $\Gamma$ has a Cayley graph of positive generalized roundness, then
$\Gamma$ must satisfy the coarse Baum-Connes conjecture, and hence the strong Novikov conjecture.
The interplay between these notions has a very interesting history. An overview is given
by Prassidis and Weston \cite{Epa}.

The set of generalized roundness exponents of a given metric space $(X, d)$ is always a closed interval of
the form $[0, \wp]$ or $[0, \infty)$, including the possibility that $\wp = 0$ in which case the interval
degenerates to $\{ 0 \}$. This result is a direct consequence of Schoenberg \cite[Theorem 2.7]{Sch}
and Lennard \textit{et al}.\ \cite[Theorem 2.4]{Ltw}. Faver \textit{et al}.\ \cite{Fav} have shown that
the interval $[0, \infty)$ arises if and only if $d$ is an ultrametric. For finite metric
spaces it is always the case that $\wp > 0$. This is the main result in Weston \cite{Wes}.

Enflo \cite{En1} constructed a separable metric space that is not uniformly embeddable in any metric
space of positive generalized roundness. Dranishnikov \textit{et al}.\ \cite{Dra} modified Enflo's
example to construct a locally finite metric space that is not coarsely embeddable in any Hilbert
space, thereby settling a prominent question of Gromov. Kelleher \textit{et al}.\ \cite{Ke1} unified these examples
to construct a locally finite metric space that is not uniformly or coarsely embeddable in any metric
space of positive generalized roundness. One may also use generalized roundness as a highly effective isometric
invariant by exploiting the connection between generalized roundness and negative type due to
Lennard \textit{et al}.\ \cite{Ltw}. A general principle for using
generalized roundness as an isometric invariant was recently isolated by Kelleher \textit{et al}.\
\cite[Theorem 3.24]{Ke2}. It is therefore a matter of great utility to be able to calculate the
generalized roundness of certain metric spaces.

In recent work, S\'{a}nchez \cite{San} has provided a method of calculating, at least numerically,
the generalized roundness of a given finite metric space $(X,d)$. However, as the size of the space
grows, S\'{a}nchez' method rapidly becomes computationally intensive. Nevertheless, the method is
an important tool for the analysis of the generalized roundness of finite metric spaces. In \cite{San},
the method is used to calculate the generalized roundness of certain finite graphs endowed with the
usual combinatorial metric. The metric graphs that we consider in this paper are countable metric
trees and so we are unable to use S\'{a}nchez' method.

It is prudent at this point to pin down some basic definitions pertaining to metric graphs.
A graph $G$ is \textit{connected} if there is a (finite) path between
any two vertices of $G$. A \textit{tree} is an undirected, connected, locally finite
graph without cycles. These definitions imply that the vertex and edge sets of a
tree are at most countable. Assigning a positive length to each edge of a given tree
$T$ induces a shortest path metric $d$ on the vertices of the tree. The resulting metric
space is denoted $(T, d)$ and is called a \textit{metric tree}.

Generalized roundness properties of metric trees have been studied by several authors.
All additive metric spaces, and hence all metric trees, have generalized roundness at least one.
This fact is folklore and it may be derived in several different ways. One such proof
appears in Kelly \cite[Theorem II]{Jbk}. Another proof follows from
Faver \textit{et al}.\ \cite[Proposition 4.1]{Fav}. Examples of Caffarelli \textit{et al}.\ \cite{Dcw}
show that some countable metric trees have generalized roundness exactly one. The situation is
different for finite metric trees. Indeed, Hjorth \textit{et al.} \cite{Hj1} have shown that
all finite metric trees have strict $1$-negative type. This condition ensures that all
finite metric trees have generalized roundness greater than one. (One way to see this is to appeal to
Lennard \textit{et al}.\ \cite[Theorem 2.4]{Ltw} and Li and Weston \cite[Corollary 4.2]{Hli}.)
Hence metric trees of generalized roundness one are necessarily countable. Simple
examples show that the converse of this statement is not true in general.

We conclude this introduction with some comments about the structure and main results
of this paper. In Section \ref{sec:2}, motivated by the local theory of Banach spaces, we
introduce a notion of finite representability for metric spaces. Our purpose in introducing
such a notion is to provide a new technique for comparing the generalized roundness of
metric spaces. The remainder of Section \ref{sec:2} is then devoted to a preliminary investigation
of this technique in the context of infinite-dimensional Banach spaces. We prove, for example,
that if $\mathcal{U}$ is any ultrafilter and $X$ is any Banach space, then the second dual
$X^{\ast\ast}$ and the ultrapower $(X)_{\mathcal{U}}$ have the same generalized roundness as $X$.
In other words, $\wp_{X} = \wp_{X^{\ast\ast}} = \wp_{(X)_{\mathcal{U}}}$. It is also noted
that no Banach space of positive generalized roundness is uniformly homeomorphic to $c_{0}$
or $\ell_{p}$, $2 < p < \infty$.

Caffarelli \textit{et al}.\ \cite{Dcw} identified several classes of metric
trees of generalized roundness one. The types of trees studied in \cite{Dcw} were spherically
symmetric, infinitely bifurcating or comb-like trees endowed with the usual combinatorial path metric.
In other words, all edges in the trees were assumed to have length one and all other distances
were determined geodesically. In Sections \ref{sec:3} and \ref{sec:4} we relax this condition
by considering trees endowed with weighted path metrics. Section \ref{sec:3} focusses on trees
that resemble jagged combs. Section \ref{sec:4} deals with spherically symmetric
trees that have systematically weighted edges. We also make a distinction between
\textit{convergent} and \textit{divergent} spherically symmetric trees. In both cases we show that
the generalized roundness of such trees can easily be one. In particular, we identify a large
class of metric trees of generalized roundness one that have finite diameter.

In Section \ref{sec:5} we examine isometric embedding properties of metric trees of generalized
roundness one. We prove that all metric trees of generalized roundness one possess the stronger
property of strict $1$-negative type. Due to the relationship between generalized roundness and
negative type, it also follows that no metric tree of generalized roundness one has $p$-negative
type for any $p > 1$. Taken together, these facts imply the following embedding phase transition:
If $(T, d)$ is a metric tree of generalized roundness one,
then (1) the metric transform $(T, \sqrt{d})$ is isometric to an affinely independent subset of $\ell_{2}$,
and (2) the metric transform $(T, \sqrt{d^{p}})$ does not embed isometrically into $\ell_{2}$ for any $p$,
$1 < p \leq 2$. Moreover, the only metric trees that satisfy condition (2) are those of generalized
roundness one.

\section{Comparing the generalized roundness of metric and Banach spaces}\label{sec:2}

In this section we develop a technique for comparing the generalized roundness of metric spaces.
In order to do this we introduce a metric space version of the Banach space notion of finite representability.
This important notion in the local theory of Banach spaces was introduced by James \cite{Ja1,Ja2}.
Throughout this section, all Banach spaces are assumed to be real and infinite-dimensional unless noted otherwise.
The first and second duals of a Banach space $X$ are denoted by $X^{\ast}$ and $X^{\ast\ast}$, respectively.
All $L_{p}$-spaces are assumed to be commutative unless noted otherwise.

\begin{definition}\label{def:0.5}
Let $X$ and $X^{\prime}$ be Banach spaces.
\begin{enumerate}
\item[(1)] $X$ is \textit{crudely represented} in $X^{\prime}$ if there exists an
$\vep_{0} > 0$ such that for each finite-dimensional subspace $E \subset X$ there exists
a finite-dimensional subspace $F \subset X^{\prime}$ (with $\dim E = \dim F$) and a
one-to-one linear mapping $T : E \rightarrow F$ that satisfies $\| T \| \| T^{-1} \| \leq 1 + \vep_{0}$.

\item[(2)] $X$ is \textit{finitely represented} in $X^{\prime}$ if for each $\vep > 0$
and each finite-dimensional subspace $E \subset X$ there exists a finite-dimensional subspace $F \subset X^{\prime}$
(with $\dim E = \dim F$) and a one-to-one linear mapping $T : E \rightarrow F$ that satisfies
\begin{align}\label{fin:rep}
(1 - \vep) \| x \| \leq \| Tx \| \leq (1 + \vep) \| x \|
\end{align}
for all $x \in E$.
\end{enumerate}
\end{definition}

It is easy to see that an equivalent reformulation of the condition given in Definition \ref{def:0.5} is the following:
for each $\vep > 0$ and each finite-dimensional subspace $E \subset X$ there exists a finite-dimensional
subspace $F \subset X^{\prime}$ (with $\dim E = \dim F$) and a one-to-one linear mapping
$T : E \rightarrow F$ that satisfies $\| T \| \| T^{-1} \| \leq 1 + \vep$. While this reformulation makes the
the relationship between crude and finite representability plain, the metric nature of Definition \ref{def:0.5} (2)
suits our purposes, not least because it motivates Definitions \ref{def:1} and \ref{def:2} below.

The notion of crude representability is particularly important in the uniform theory of Banach spaces.
Recall that two Banach spaces $X$ and $X^{\prime}$ are said to be \textit{uniformly homeomorphic}
if there exists a bijection $f: X \rightarrow X^{\prime}$ such that $f$ and $f^{-1}$ are both
uniformly continuous. A famous result of Ribe \cite{Rib} asserts that if a Banach space $X$ is
uniformly homeomorphic to a Banach space $X^{\prime}$, then $X$ is crudely represented in $X^{\prime}$
and $X^{\prime}$ is crudely represented in $X$. This result is sometimes known as Ribe's
rigidity theorem.

A one-to-one linear mapping $T : E \rightarrow F$ that satisfies condition (\ref{fin:rep}) is said to be a
\textit{$(1 + \vep)$-isomorphism}. A similar notion for metric spaces may be formulated as follows.

\begin{definition}\label{def:1}
Let $(X,d)$ and $(X^{\prime}, \rho)$ be metric spaces, and suppose that $\vep > 0$. A one-to-one mapping
$\phi : X \rightarrow X^{\prime}: x \mapsto x^{\prime}$ is called a \textit{$(1 + \vep)$-scale isomorphism}
if there exists a constant $n = n(\vep) > 0$ such that
\[
(1-\vep)n d(a,b) \leq \rho(a^{\prime},b^{\prime}) \leq (1 + \vep)n d(a,b)
\]
for all $a,b \in X$.
\end{definition}

It is worth noting that we will use the notation $x^{\prime}$ to denote $\phi(x)$ throughout this section.

\begin{definition}\label{def:2}
A metric space $(X,d)$ is said to be \textit{locally represented} in a metric space $(X^{\prime}, \rho)$
if for each $\vep > 0$ and each non-empty finite set $X^{\sharp} \subseteq X$ there exists a $(1 + \vep)$-scale
isomorphism $\phi : (X^{\sharp}, d) \rightarrow (X^{\prime}, \rho)$.
\end{definition}

The following lemma notes that for Banach spaces, finite representation implies local representation.

\begin{lemma}\label{lem:0}
Let $X$ and $X^{\prime}$ be given Banach spaces.
If $X$ is finitely represented in $X^{\prime}$, then $X$ is locally represented in $X^{\prime}$.
\end{lemma}

\begin{proof}
Let $X^{\sharp}$ be a given non-empty finite subset of $X$ and suppose that $\vep > 0$.
Let $E$ denote the linear span of $X^{\sharp}$ in $X$. Then $E$ is a finite-dimensional subspace of $X$.
As $X$ is finitely represented in $X^{\prime}$, there exists a $(1 + \vep)$-isomorphism
$T : E \rightarrow X^{\prime}$. Setting $\phi$ to be the restriction of $T$ to $X^{\sharp}$
we obtain a $(1 + \vep)$-scale isomorphism $X^{\sharp} \rightarrow X^{\prime}$ (with constant $n = 1$).
Hence then $X$ is locally represented in $X^{\prime}$.
\end{proof}

We turn now to the main technical result of this section. It provides a new technique for comparing
the generalized roundness of metric spaces.

\begin{theorem}\label{thm:1}
If a metric space $(X,d)$ is locally represented in a metric space $(X^{\prime}, \rho)$, then
every generalized roundness exponent of $(X^{\prime}, \rho)$ is a generalized roundness
exponent of $(X,d)$. Hence, $\wp_{X^{\prime}} \leq \wp_{X}$.
\end{theorem}

\begin{proof}
It suffices to prove that if $p$ is not a generalized roundness
exponent of $(X, d)$ then $p$ is not a generalized roundness of $(X^{\prime}, \rho)$.

Suppose that $p \geq 0$ is not a generalized roundness exponent of $(X, d)$. We immediately have
that $p > 0$ because $0$ is a generalized roundness exponent of all metric spaces. From our definition there must
be a simplex $[a_{i};b_{j}] \subseteq X$ such that
\begin{eqnarray*}
\sum\limits_{ i < j} \left( d(a_{i}, a_{j})^{p} + d(b_{i}, b_{j})^{p} \right)
& > & \sum\limits_{i,j} d(a_{i}, b_{j})^{p}.
\end{eqnarray*}
The limiting behavior $y = x^{p}$ in a neighborhood of $x = 1$ then ensures that
we may choose an $\vep > 0$ so that
\begin{eqnarray}\label{ineq:1}
(1 - \vep)^{p} \cdot\sum\limits_{ i < j} \left( d(a_{i}, a_{j})^{p} + d(b_{i}, b_{j})^{p} \right)
& > & (1 + \vep)^{p} \cdot \sum\limits_{i,j} d(a_{i}, b_{j})^{p}.
\end{eqnarray}

We now let $X^{\sharp}$ denote the finite subset of $X$ that consists of the simplex points $a_{i}, b_{j}$.
As $(X,d)$ is locally represented in $(X^{\prime}, \rho)$ and $\vep > 0$, there must exist an
injection $\phi: X^{\sharp} \rightarrow X^{\prime}: x \mapsto x^{\prime}$ and a constant $n = n(\vep) > 0$ such that
\begin{align*}
(1 - \vep) n d(a,b) \leq \rho(a^{\prime},b^{\prime}) \leq (1 + \vep)n d(a,b)
\end{align*}
for all $a,b \in X^{\sharp}$. Furthermore, if we scale the metric on $X^{\prime}$ by defining
$\omega = \rho / n$, we immediately obtain
\begin{align}\label{ineq:2}
(1 - \vep) d(a,b) \leq \omega(a^{\prime},b^{\prime}) \leq (1 + \vep) d(a,b)
\end{align}
for all $a,b \in X^{\sharp}$. It now follows from (\ref{ineq:1}) and (\ref{ineq:2}) that $p$
is not a generalized roundness exponent for the scaled metric space $(X^{\prime}, \omega)$. Indeed,
\begin{eqnarray*}
\sum\limits_{ i < j} \left( \omega(a_{i}^{\prime}, a_{j}^{\prime})^{p} + \omega(b_{i}^{\prime}, b_{j}^{\prime})^{p} \right)
& \geq & (1 - \vep)^{p} \cdot\sum\limits_{ i < j} \left( d(a_{i}, a_{j})^{p} + d(b_{i}, b_{j})^{p} \right) \\
&   >  & (1 + \vep)^{p} \cdot \sum\limits_{i,j} d(a_{i}, b_{j})^{p} \\
&   =  & \sum\limits_{i,j} \left( (1 + \vep) d(a_{i}, b_{j}) \right)^{p} \\
& \geq & \sum\limits_{i,j} \omega(a_{i}^{\prime}, b_{j}^{\prime})^{p}.
\end{eqnarray*}
This completes the proof because generalized roundness is preserved under any scaling of the metric $\rho$.
\end{proof}

For the remainder of this section we will consider the application of Theorem \ref{thm:1} to Banach spaces.
It is germane to recall a few facts about the generalized roundness of $L_{p}$-spaces. If $X$ is an $L_{p}$-space,
then $\wp_{X} = p$ if $1 \leq p \leq 2$ and $\wp_{X} = 0$ if $p > 2$. These results are due to Enflo \cite{En1}
in the case $1 \leq p \leq 2$ and Lennard \textit{et al}.\ \cite{Ltw} in the case $p > 2$. With the exception
of the Schatten $p$-classes $C_{p}$, the generalized roundness of non-commutative $L_{p}$-spaces has not been
widely studied. In \cite{Ltw} the authors noted that $\wp_{C_{p}} = 0$ if $p > 2$. It is only relatively recently
that Dahma and Lennard \cite{Dah} have shown that $\wp_{C_{p}} = 0$ if $0 < p < 2$.


\begin{corollary}\label{cor:1}
If a Banach space $X$ is finitely represented in a Banach space $X^{\prime}$, then $\wp_{X^{\prime}} \leq \wp_{X}$.
In particular, if $\wp_{X} = 0$, then $\wp_{X^{\prime}} = 0$.
\end{corollary}

\begin{proof}
Immediate from Lemma \ref{lem:0} and Theorem \ref{thm:1}.
\end{proof}

Examples of Banach spaces that have generalized roundness zero include $C[0,1]$,
$\ell_{\infty}$, $c_{0}$, $\ell_{p}$ if $p > 2$, and the Schatten $p$-class $C_{p}$ if $p \not= 2$.
For each Banach space $X$ and each integer $n \geq 2$,
Dineen \cite{Din} has shown that $\ell_{\infty}$ is finitely represented in the space $\mathcal{P}(\prescript{n}{}X)$
of bounded $n$-homogenous polynomials on $X$.
Hence $\mathcal{P}(\prescript{n}{}X)$ has generalized roundness zero by Corollary \ref{cor:1}.
For each $p \in (1, \infty)$, $c_{0}$ is finitely represented in the quasi-reflexive James space $J_{p}$.
(This result is due to Giesy and James \cite{Gie} in the case $p = 2$ and, for $p \not= 2$, it is due to
Bird \textit{et al}.\ \cite{Bir}.) Hence, for each $p \in (1, \infty)$, $J_{p}$ has generalized roundness zero
by Corollary \ref{cor:1}.

On the basis of existing theory and Corollary \ref{cor:1} we are able to isolate some situations
where generalized roundness functions as an invariant in the uniform theory of Banach spaces. For
instance, as the next corollary shows, no Banach space of positive generalized roundness is uniformly
homeomorphic to $\ell_{p}$ for any $p > 2$.

\begin{corollary}\label{cor:1.5}
If a Banach space $X$ is uniformly homeomorphic to $\ell_{p}$ ($1 \leq p < \infty$),
then $\wp_{X} \leq \wp_{\ell_{p}}$. In particular, if $p > 2$, then $\wp_{X} = 0$.
\end{corollary}

\begin{proof}
Ribe's rigidity theorem \cite{Rib} implies that $\ell_{p}$ is crudely represented in $X$.
However, if $\ell_{p}$ is crudely represented in $X$, then $\ell_{p}$
is finitely represented in $X$. This follows from Krivine's theorem \cite{Kri}
(as noted by Rosenthal \cite{Ros} and Lemberg \cite{Lem}) if $1 < p < \infty$,
and it is due to James \cite{Ja0} in the case $p = 1$.
Thus $\wp_{X} \leq \wp_{\ell_{p}}$ by Corollary \ref{cor:1}. Moreover, if $p > 2$,
then $\wp_{\ell_{p}} = 0$. So for $p > 2$ we deduce that $\wp_{X} = 0$.
\end{proof}

The uniform structure of $\ell_{p}$, $1 < p < \infty$, is particularly well-understood. For instance,
if a Banach space $X$ is uniformly homeomorphic to $\ell_{p}$, $1 < p < \infty$, then it is linearly
isomorphic to $\ell_{p}$. This deep theorem is due to Enflo \cite[Theorem 6.3.1]{En2} in the case
$p = 2$ and Johnson \textit{et al}.\ \cite[Theorem 2.1]{Joh} when $p \not= 2$. So if $1 < p < \infty$,
one may replace the phrase ``uniformly homeomorphic'' in the statement of Corollary \ref{cor:1.5} with
the phrase ``linearly isomorphic'' without losing any generality. The situation for $c_{0}$ is somewhat similar.

\begin{corollary}
If a Banach space $X$ is uniformly homeomorphic to $c_{0}$, then $\wp_{X} = 0$.
In particular, no Banach space of positive generalized roundness is uniformly
homeomorphic to $c_{0}$.
\end{corollary}

\begin{proof}
Ribe's rigidity theorem \cite{Rib} implies that $c_{0}$ is crudely represented in $X$.
However, James \cite{Ja0} has shown that if $c_{0}$ is crudely represented in $X$,
then $c_{0}$ is finitely represented in $X$. Thus $\wp_{X} \leq \wp_{c_{0}}$ by Corollary \ref{cor:1}.
Moreover, as $\wp_{c_{0}} = 0$, we further deduce that $\wp_{X} = 0$.
\end{proof}

The uniform structure of $c_{0}$ is more beguiling and less well understood than that of $\ell_{p}$,
$1 < p < \infty$. Johnson \textit{et al}.\ \cite[Corollary 3.2]{Joh} proved that if a complemented
subspace of a $C(K)$ space is uniformly homeomorphic to $c_{0}$, then it is linearly isomorphic to $c_{0}$.
Godefroy \textit{et al}.\ \cite[Theorem 5.6]{God} have shown that a Banach space which is uniformly
homeomorphic to $c_{0}$ is an isomorphic predual of $\ell_{1}$ with summable Szlenk index. But it is not known
whether a predual of $\ell_{1}$ with summable Szlenk index is linearly isomorphic to $c_{0}$. Thus, unlike $\ell_{p}$
($1 < p < \infty$), it remains unclear whether $c_{0}$ is determined by its uniform structure.

In the proof of the following corollary we invoke the Principle of Local Reflexivity. This principle
is central to the local theory of Banach spaces and it is originally due to Lindenstrauss and Rosenthal \cite{Lin}.

\begin{corollary}\label{cor:2}
If $X$ is a Banach space, then $\wp_{X} = \wp_{X^{\ast\ast}}$ and $\wp_{X^{\ast}} = \wp_{X^{\ast\ast\ast}}$.
\end{corollary}

\begin{proof}
As $X$ embeds isometrically into $X^{\ast\ast}$ we see that $\wp_{X\ast\ast} \leq \wp_{X}$.
In addition, the Principle of Local Reflexivity implies that $X^{\ast\ast}$
is finitely represented in $X$. Hence $\wp_{X} \leq \wp_{X^{\ast\ast}}$ by Corollary \ref{cor:1}. By
combining these two inequalities we obtain $\wp_{X} = \wp_{X^{\ast\ast}}$. By replacing $X$
with $X^{\ast}$ we also see that $\wp_{X^{\ast}} = \wp_{X^{\ast\ast\ast}}$.
\end{proof}

Examples show that for a Banach space $X$ we may have $\wp_{X} \not= \wp_{X^{\ast}}$. Indeed,
if $p \in [1, 2)$ and $X = \ell_{p}$, then $\wp_{X} = p$ and $\wp_{X^{\ast}} = 0$. By way of comparison,
if $p \not= 2$ and $X = C_{p}$, then $\wp_{X} = \wp_{X^{\ast}} = 0$. Thus, given Banach space $X$,
the entries of the sequence $(\wp_{X}, \wp_{X^{\ast}}, \wp_{X^{\ast\ast}}, \ldots)$ are restricted
to take on at most two values (in the interval $[0,2]$) by Corollary \ref{cor:2}.

Intimately related to the concept of finite representability is the notion of an ultrapower of a
Banach space. Given an ultrafilter $\mathcal{U}$ on a set $I$ and a Banach space $X$ there is a
canonical procedure to construct a large Banach space $(X)_{\mathcal{U}}$ called the \textit{ultrapower}
of $X$. Importantly, $(X)_{\mathcal{U}}$ contains a natural isometric copy of $X$ and it is finitely
represented in $X$. For a detailed construction of $(X)_{\mathcal{U}}$, and a discussion of the interplay
between finite representability and ultrapowers, we refer the reader to H\'{a}jek and Johanis \cite{Haj}.

\begin{corollary}\label{cor:3}
Let $\mathcal{U}$ be a given ultrafilter on a set $I$ and let $X$ be a Banach space.
Then $\wp_{X} = \wp_{(X)_{\mathcal{U}}}$.
\end{corollary}

\begin{proof}
As $X$ embeds isometrically into $(X)_{\mathcal{U}}$ we see that $\wp_{(X)_{\mathcal{U}}} \leq \wp_{X}$.
In addition, $(X)_{\mathcal{U}}$ is finitely represented in $X$ by Stern \cite[Theorem 6.6]{Ste}.
Hence $\wp_{X} \leq \wp_{(X)_{\mathcal{U}}}$ by Corollary \ref{cor:1}.
By combining these two inequalities we obtain $\wp_{X} = \wp_{(X)_{\mathcal{U}}}$.
\end{proof}

Lennard \textit{et al}.\ \cite[Theorem 2.3]{Lt2} noticed that if the infimal cotype of a Banach space
$X$ is greater than two, then $X$ must have generalized roundness zero. By utilizing deep theory and
Corollary \ref{cor:1} we are able to exhibit a more precise relationship between the supremal type and
infimal cotype of a Banach space and its generalized roundness. The notions of type and cotype have
been paramount in the local theory of Banach spaces for quite some time and are defined in the following manner.

\begin{definition}\label{def:type}
A Banach space $X$ is said to have \textit{type} $p$ if there exists a constant
$\mathcal{A} \in (0,\infty)$ such that for all integers $n > 0$ and for all finite sequences
$(x_{j})_{j=1}^{n}$ in $X$, we have
\begin{eqnarray}\label{type:inq}
\sum\limits_{\epsilon \in \{-1,+1\}^{n}}
\left\|\sum\limits_{j=1}^{n} \frac{\epsilon_{j}x_{j}}{2^{n}} \right\|_{X} & \leq &
\mathcal{A}
\left(\sum\limits_{j=1}^{n}\|x_{j}\|_{X}^{p}\right)^{\frac{1}{p}}.
\end{eqnarray}
\textit {Cotype} $p$ is defined similarly but with the inequality (\ref{type:inq}) reversed.
\end{definition}

It is well-known that no Banach space can have type $p > 2$ or cotype $q < 2$.
We let $p(X)$ denote the supremum of all $p$ such that $X$ has type $p$ and
$q(X)$ denote the infimum of all $q$ such that $X$ has cotype $q$.
For an overview of theory of type and cotype we refer the reader to Diestel \textit{et al}.\ \cite{Djt}.

A famous theorem of Maurey and Pisier \cite{Mau} states that $\ell_{p(X)}$
and $\ell_{q(X)}$ are finitely represented in $X$. This theorem and Corollary \ref{cor:1}
provide an immediate link to generalized roundness.

\begin{corollary}\label{cor:4}
If $X$ is a Banach space, then $\wp_{X} \leq \min \{ \wp_{\ell_{p(X)}}, \wp_{\ell_{q(X)}} \}$.
In particular, if $q(X) > 2$, then $\wp_{X} = 0$.
\end{corollary}

\begin{proof}
By the Maurey-Pisier theorem, $\ell_{p(X)}$ and $\ell_{q(X)}$ are finitely represented in $X$.
Hence $\wp_{X} \leq \wp_{\ell_{p(X)}}$ and $\wp_{X} \leq \wp_{\ell_{q(X)}}$ by Corollary \ref{cor:1}.
In particular, if $q(X) > 2$, then $\wp_{\ell_{q(X)}} = 0$, and so $\wp_{X} = 0$.
\end{proof}

There are some classical Banach spaces for which the inequality in Corollary \ref{cor:4} is an
equality. For example, if $X$ is an $L_{p}$-space, $1 \leq p < \infty$, then
$\wp_{X} = \min \{ \wp_{\ell_{p(X)}}, \wp_{\ell_{q(X)}} \}$. In this case, $\wp_{X} = p$ if
$1 \leq p \leq 2$ and $\wp_{X} = 0$ if $p > 2$. Moreover, it is well known that
$p(X) = \min \{ p, 2\}$ and $q(X) = \max \{ p, 2 \}$. So, for example, if $p > 2$, then
$\wp_{X} = 0$ and $q(x) = p$. Thus $\wp_{\ell_{q(X)}} = 0$.
On the other hand, if $X = C_{p}$, $1 \leq p < \infty$, then $p(X)$ and $q(X)$ have the same values as
any $L_{p}$-space but, by inspection, $\wp_{X} = \min \{ \wp_{\ell_{p(X)}}, \wp_{\ell_{q(X)}} \}$
if and only if $p \geq 2$.

\section{Comb-like graphs of generalized roundness one}\label{sec:3}

In this section we apply Theorems \ref{thm:1} and \ref{thm:2} to analyze the generalized roundness
of countable metric trees that resemble combs. We first give
sufficient conditions for the existence of a $(1 + \vep)$-scale isomorphism $\phi : (T,d) \rightarrow (T,\rho)$,
under the assumption that $d$ and $\rho$ are path weighted metrics on a given finite tree $T$.
In what follows, we let $\mathbb{N}$ denote the set of all non-negative integers. Moreover, given a positive
integer $m$, we let $[m]$ denote the segment $\{0, 1, 2, \ldots, m \}$.

\begin{lemma}\label{lem:2}
Let $d$ and $\rho$ be two path weighted metrics on a given finite tree $T$. Let
\begin{align*}
m = \max \left\{ \frac{\rho(\mathbf{a},\mathbf{b})}{d(\mathbf{a},\mathbf{b})} :
\mathbf{a}, \mathbf{b} \in T \text{ and } \mathbf{a} \not= \mathbf{b} \right\}.
\end{align*}
Then there must be an edge $\{ \mathbf{x}, \mathbf{y} \}$ in $T$ such that
$m = \frac{\rho(\mathbf{x}, \mathbf{y})}{d(\mathbf{x}, \mathbf{y})}$.
\end{lemma}

\begin{proof}
Suppose that $\mathbf{a}, \mathbf{c} \in T$ are non-adjacent vertices such that
$m = \rho(\mathbf{a}, \mathbf{c})/d(\mathbf{a}, \mathbf{c})$.
Then we may choose a strictly intermediate vertex $\mathbf{b} \in T$ on the geodesic from $\mathbf{a}$ to $\mathbf{c}$.
Now let $q = \rho(\mathbf{a}, \mathbf{b})/d(\mathbf{a}, \mathbf{b})$ and
$r = \rho(\mathbf{b}, \mathbf{c})/d(\mathbf{b}, \mathbf{c})$. Without loss of generality
we may assume that $q \geq r$. Furthermore, as $\rho$ and $d$ are path metrics on $T$, we
have $\rho(\mathbf{a}, \mathbf{c}) = \rho(\mathbf{a}, \mathbf{b}) + \rho(\mathbf{b}, \mathbf{c})$
and $d(\mathbf{a}, \mathbf{c}) = d(\mathbf{a}, \mathbf{b}) + d(\mathbf{b}, \mathbf{c})$. In particular,
it follows that
\begin{eqnarray*}
m &   =  & \frac{\rho(\mathbf{a}, \mathbf{c})}{d(\mathbf{a}, \mathbf{c})} \\
  &   =  & \frac{\rho(\mathbf{a}, \mathbf{b}) + \rho(\mathbf{b}, \mathbf{c})}{d(\mathbf{a}, \mathbf{b}) + d(\mathbf{b}, \mathbf{c})} \\
  &   =  & \frac{qd(\mathbf{a}, \mathbf{b}) + rd(\mathbf{b}, \mathbf{c})}{d(\mathbf{a}, \mathbf{b}) + d(\mathbf{b}, \mathbf{c})} \\
  & \leq & \frac{qd(\mathbf{a}, \mathbf{b}) + qd(\mathbf{b}, \mathbf{c})}{d(\mathbf{a}, \mathbf{b}) + d(\mathbf{b}, \mathbf{c})} \\
  &   =  & q.
\end{eqnarray*}
Therefore, by definition of $m$, it must be the case that $m = q$. This shows that we can always
pass to a pair of vertices connected by a geodesic with fewer edges and preserve the ratio $m$.
Applying this logic finitely many times gives the lemma.
\end{proof}

The following analogous lemma for minima may be proved in the same way.

\begin{lemma}\label{lem:2.5}
Let $d$ and $\rho$ be two path weighted metrics on a given finite tree $T$. Let
\begin{align*}
m^{\ast} = \min \left\{ \frac{\rho(\mathbf{a},\mathbf{b})}{d(\mathbf{a},\mathbf{b})} :
\mathbf{a}, \mathbf{b} \in T \text{ and } \mathbf{a} \not= \mathbf{b} \right\}.
\end{align*}
Then there must be an edge $\{ \mathbf{x}, \mathbf{y} \}$ in $T$ such that
$m^{\ast} = \frac{\rho(\mathbf{x}, \mathbf{y})}{d(\mathbf{x}, \mathbf{y})}$.
\end{lemma}

\begin{theorem}\label{thm:2}
Let $\vep > 0$ be given. Let $d$ and $\rho$ be two path weighted metrics on a given finite tree $T$.
If there exists a constant $n = n(\vep) > 0$ such that
\begin{align}\label{asi:inq}
(1 - \vep)n \leq \frac{\rho(\mathbf{x}, \mathbf{y})}{d(\mathbf{x}, \mathbf{y})} \leq n(1 + \vep)
\end{align}
for each edge $\{ \mathbf{x}, \mathbf{y} \}$ in $T$, then the identity map
$\phi : (T,d) \rightarrow (T,\rho): x \mapsto x$ is a $(1 + \vep)$-scale isomorphism.
\end{theorem}

\begin{proof}
Using the notation of Lemmas \ref{lem:2} and \ref{lem:2.5} it follows from (\ref{asi:inq})
that
\begin{align*}
(1 - \vep)n \leq m^{\ast} \leq m \leq n(1 + \vep).
\end{align*}
Thus, given any two distinct vertices $\mathbf{a}, \mathbf{b} \in T$,
we deduce that $(1 - \vep)n \leq \rho(\mathbf{a}, \mathbf{b})/d(\mathbf{a}, \mathbf{b}) \leq n(1 + \vep)$
by definition of $m$ and $m^{\ast}$.
Hence the identity map $\phi : (T,d) \rightarrow (T,\rho): x \mapsto x$ is a $(1 + \vep)$-scale isomorphism.
\end{proof}

We now apply Theorems \ref{thm:1} and \ref{thm:2} to analyze the generalized roundness
of certain countable metric trees that resemble combs.

\begin{definition}\label{combs}
The vertex set $V$ of the \textit{infinite comb} $C$ consists of the points $\mathbf{x}_{0}$, $\mathbf{x}_{k}$
and $\mathbf{y}_{k}$, where $k$ is any positive integer. The edge set $E$ of $C$ consists of the
unordered pairs $\{ \mathbf{x}_{k-1}, \mathbf{x}_{k} \}$ and $\{ \mathbf{x}_{k}, \mathbf{y}_{k} \}$,
where $k$ is any positive integer.

For each positive integer $m$, the finite subtree of $C$ that has vertex set
$\{ \mathbf{x}_{1+k}, \mathbf{y}_{1+k} \,|\, k \in [m] \}$ will be called the $m$-\textit{comb}.
The $m$-comb will be denoted $C_{m}$.
\end{definition}

We are interested in placing various path metrics on the infinite comb $C$ and the $m$-comb $C_{m}$.
One way to do this is to adopt the following canonical procedure.

\begin{definition}\label{ice:def}
Let $f : \mathbb{N} \rightarrow (0, \infty)$ be a function.
We define a path metric $\rho_{f}$ on the infinite comb $C$ in the following manner:
\begin{enumerate}
\item[(1)] $\rho_{f}(\mathbf{x}_{k-1},\mathbf{x}_{k}) = f(k-1)$, and

\item[(2)] $\rho_{f}(\mathbf{x}_{k},\mathbf{y}_{k}) = f(k)$ for each positive integer $k$.
\end{enumerate}
All other distances in $C$ are then determined geodesically. The resulting metric tree will be denoted $C(f)$.
If, moreover, we restrict $\rho_{f}$ to the $m$-comb, the resulting metric tree will be denoted $C_{m}(f)$.
\end{definition}

There are some special cases of Definition \ref{ice:def} worth highlighting.
If $f(k) = 1$ for all $k \geq 0$, the resulting metric trees $C(f)$ and $C_{m}(f)$ will
be denoted $C(1)$ and $C_{m}(1)$, respectively. In other words, $C(1)$ and $C_{m}(1)$ are the combs $C$ and $C_{m}$
endowed with the usual combinatorial path metric $\delta$. Caffarelli \textit{et al}.\ \cite{Dcw} have
shown that $\wp_{C_{m}(1)} \rightarrow 1$ as $m \rightarrow \infty$. Hence $\wp_{C(1)} = 1$.
We will see presently that by placing mild assumptions on the function $f$ it follows that $\wp_{C(f)} = 1$.
Our arguments will be facilitated by the following lemma.

\begin{lemma}\label{lem:3}
If the $m$-comb $C_{m}(1)$ is locally represented in a metric tree $(T, d)$ for all integers $m > 0$,
then $\wp_{(T, d)} = 1$.
\end{lemma}

\begin{proof}
All metric trees have generalized roundness at least one.
By Theorem \ref{thm:1}, $\wp_{(T, d)} \leq \wp_{C_{m}(1)}$ for all $m > 0$.
Moreover, Caffarelli \textit{et al}.\ \cite{Dcw} have shown that $\wp_{C_{m}(1)} \rightarrow 1$
as $m \rightarrow \infty$. Hence $\wp_{(T, d)} = 1$.
\end{proof}

\begin{definition}
A function $f : \mathbb{N} \rightarrow (0, \infty)$ is said to be \textit{additively sub-exponential} if
$\lim\limits_{n \rightarrow \infty} \frac{f(n+m)}{f(n)}= 1$ for each integer $m > 0$.
\end{definition}

The class of additively sub-exponential functions $f : \mathbb{N} \rightarrow (0, \infty)$ is very large.
For instance, $f$ could be any rational function that takes positive values on $\mathbb{N}$. Other interesting
possibilities for $f$ include inverse tangent, logarithmic functions (translated suitably), and classically
sub-exponential functions such as $e^{\sqrt{n}}$. Furthermore, if $f : \mathbb{N} \rightarrow (0, \infty)$
is an additively sub-exponential function, then so is $1/f$.

\begin{theorem}\label{thm:3}
If $f : \mathbb{N} \rightarrow (0, \infty)$ is additively sub-exponential function, then $\wp_{C(f)} = 1$.
\end{theorem}

\begin{proof} By Lemma \ref{lem:3}, it suffices to prove that $C_{m}(1)$ is locally represented
in $C(f)$ for all $m > 0$.

Let $m > 0$ be a given integer. Let $\vep > 0$ be given. As $f$ is additively sub-exponential, we may
choose an integer $n_{0} > 0$ so that $1 - \vep \leq f(n+k)/f(n) \leq 1 + \vep$ for each $k \in [m]$ and all
$n \geq n_{0}$. In particular, we have
\begin{align*}
1 - \vep \leq \frac{f(n_{0}+k)}{f(n_{0})} \leq 1 + \vep
\end{align*}
for each $k \in [m]$.
Consider the subtree $Y^{\prime}$ of $C(f)$ that has vertices $\mathbf{x}_{n_{0} + k}$
and $\mathbf{y}_{n_{0} + k}$ for all $k \in [m]$. As simple (unweighted) graphs,
$Y^{\prime}$ and $C_{m}(1)$ are one and the same graph; namely, $C_{m}$.
Let $\phi : C_{m}(1) \rightarrow Y^{\prime}$ denote this natural identification.
Let $\rho$ denote the path metric that $Y^{\prime}$ inherits from $C(f)$.
We may regard the metrics on $Y^{\prime}$ and $C_{m}(1)$ as path metrics on $C_{m}$.
For each edge $\{ \mathbf{s}, \mathbf{t} \}$ in $C_{m}$ we have, by choice of $n_{0}$,
\begin{align*}
f(n_{0})(1 - \vep) \leq \min\limits_{k \in [m]} f(n_{0} + k)\leq
\frac{\rho(\phi(\mathbf{s}), \phi(\mathbf{t}))}{\delta(\mathbf{s}, \mathbf{t})} \leq
\max\limits_{k \in [m]} f(n_{0} + k) < f(n_{0})(1 + \vep).
\end{align*}
It follows from Theorem \ref{thm:2} that $\phi : C_{m}(1) \rightarrow Y^{\prime}$ is a $(1 + \vep)$-scale
isomorphism.
As $C_{m}(1)$ is a finite metric space and as $\vep > 0$ was arbitrary, we conclude that
$C_{m}(1)$ is locally represented in $C(f)$.
\end{proof}

It also follows from Theorem \ref{thm:3} that if $f$ is an additively sub-exponential function, then $\wp_{C(f)} = \wp_{C(1/f)}$.

\section{Convergent and divergent spherically symmetric trees of generalized roundness one}\label{sec:4}

Caffarelli \textit{et al}.\ \cite{Dcw} considered the generalized roundness of spherically symmetric
trees endowed with the usual combinatorial path metric (wherein all edges in the tree are assumed to have unit length).
In this section we consider a broader class of spherically symmetric trees by relaxing the requirement that all
edges have unit length. This allows one to make a distinction between \textit{convergent} and \textit{divergent} spherically
symmetric trees. In order to proceed we review the basic definitions and notations for spherically symmetric
trees. In addition, we introduce the notion of a downward length sequence for a spherically symmetric tree.

Given a vertex $\mathbf{v}_{0}$ in a tree $T$, we let
$r(T, \mathbf{v}_{0}) = \sup \{ \delta(\mathbf{v}_{0}, \mathbf{v}) : \mathbf{v} \in T \}$, where $\delta$
denotes the usual combinatorial path metric on $T$. We call $r(T, \mathbf{v}_{0})$ the
$\mathbf{v}_{0}$-\textit{depth} of $T$. Naturally included here is the possibility that the
$\mathbf{v}_{0}$-depth of $T$ may be infinite.
A vertex $\mathbf{v}$ of $T$ is a \textit{level} $k$ vertex of $T$ if $\delta(\mathbf{v}_{0}, \mathbf{v}) = k$.
The \textit{children} of a level $k$ vertex $\mathbf{v} \in T$ consist of all level $k + 1$ vertices $\mathbf{w} \in T$
such that $\delta(\mathbf{v}, \mathbf{w}) = 1$. We let $d_{k}(\mathbf{v})$ denote the number of children of $\mathbf{v}$.

\begin{definition}\label{s:tree}
A tree $T$ is said to be \textit{spherically symmetric} if we can choose a vertex $\mathbf{v}_{0} \in T$ so that
for any $k$, all level $k$ vertices of $T$ have the same number of children.
Such a pair $(T, \mathbf{v}_{0})$ will be called a \textit{spherically symmetric tree} (SST).
\end{definition}

Notice that if $\mathbf{v}$ is a level $k$ vertex in a given SST $(T, \mathbf{v}_{0})$, then
$d_{k} = d_{k}(\mathbf{v})$ only depends upon $k$. Thus $d_{k}$ is the number of children
of any level $k$ vertex in $T$. We call the (possibly finite) sequence $(d_k)_{0 \le k < r(T, \mathbf{v}_{0})}$
the \textit{downward degree sequence} of $(T, \mathbf{v}_{0})$. We will say that $(d_k)$ is \textit{non-trivial}
provided $d_{k} > 1$ for at least one $k$ such that $0 \leq k < r(T, \mathbf{v}_{0})$.

Now suppose that $(T, \mathbf{v}_{0})$ is a given SST with downward degree sequence $(d_{k})$.
If, for each $k$ such that $0 \le k < r(T, \mathbf{v}_{0})$, $l_{k}$ is a positive real number, we will call
$\ell = (l_{0}, l_{1}, l_{2}, \ldots)$ a \textit{downward length sequence} for $(T, \mathbf{v}_{0})$.
Given such a sequence $\ell$, we may define a path metric $\rho_{\ell}$ on $T$ in the following manner.
For any $k$ such that $0 \leq k < r(T, \mathbf{v}_{0})$,
if $\mathbf{w}$ is a child of a level $k$ vertex $\mathbf{v} \in T$, we define $\rho_{\ell}(\mathbf{v}, \mathbf{w}) = l_{k}$.
All other $\rho_{\ell}$-distances in $T$ are then determined geodesically. The resulting metric tree
$(T, \rho_{\ell})$ is said to be \textit{convergent} if $\sum l_{k} < \infty$ and \textit{divergent} if
$\sum l_{k} = \infty$. One significance of convergent SSTs is that they have finite diameter.

We proceed to show that large classes of divergent and convergent SSTs have generalized roundness one.
The following lemma is a variation of \cite[Theorem 2.1]{Dcw}. As the statement of the lemma
is complicated, we will comment on the intuition behind this result. Among all $n$-point
metric trees endowed with the usual combinatorial path metric, the complete bipartite graph $K_{1, n-1}$
has the smallest generalized roundness. Moreover, as $n \rightarrow \infty$, the generalized roundness
of $K_{1, n-1}$ tends to one. The conditions placed on the SST in the statement of the following lemma
ensure that it contains a star-like structure that resembles $K_{1, q}$, $q = d_{0}d_{1}\cdots d_{k}$,
modulo scaling. Such an SST must have generalized roundness relatively close to one.

\begin{lemma}\label{sst}
Let $(T, \mathbf{v}_{0})$ be a finite SST with a non-trivial downward degree sequence $(d_{0}, d_{1}, \ldots, d_{n-1})$.
Suppose $\ell = (l_{0}, l_{1}, \ldots, l_{n-1})$ is a downward length sequence for $(T, \mathbf{v}_{0})$ that
satisfies $2l_{0} < l_{0} + l_{1} + \cdots + l_{n-1}$. For each $k$, $1 \leq k \leq n$, set
$M_{k} = \sum_{i=0}^{k-1} l_{i}$ and let $m$ be the largest integer $k$ such that $M_{k} < \frac{1}{2} M_{n}$.
For each non-negative integer $k \leq m$ such that $d_{0}d_{1}\cdots d_{k} > 1$, we have
\begin{eqnarray}
\wp_{(T, \rho_{\ell})} & \leq &
\frac{\ln\left(2+\frac{2}{(d_{0}d_{1}\cdots d_{k}) - 1}\right)}{\ln\left(2-\frac{2M_{k}}{M_{n}}\right)}.
\end{eqnarray}
If $d_{0}d_{1}\cdots d_{m} = 1$, then we have the trivial bound $\wp_{(T, \rho_{\ell})} \leq 2$.
\end{lemma}

\begin{proof}
Let $(T, \mathbf{v}_{0})$ be a finite SST that satisfies the hypotheses of the lemma.
Because at least one $d_{j} > 1$ there must exist at least one vertex $\mathbf{z} \in T$
with at least two children $\mathbf{x}, \mathbf{y} \in T$. As children in $T$ are $\rho_{\ell}$-equidistant
from their parent we see that
\begin{align*}
\rho_{\ell}(\mathbf{x}, \mathbf{z}) = \frac{\rho_{\ell}(\mathbf{x}, \mathbf{y})}{2} = \rho_{\ell}(\mathbf{z}, \mathbf{y}).
\end{align*}
The existence of such a metric midpoint in $(T, \rho_{\ell})$ ensures that $\wp_{(T, \rho_{\ell})} \leq 2$.

Now assume that $d_{0}d_{1}\cdots d_{m} > 1$ and consider any non-negative integer $k \leq m$
such that $d_{0}d_{1}\cdots d_{k} > 1$. Then there are $d_{0}d_{1} \cdots d_{k-1}$
vertices at distance $M_{k}$ from $\mathbf{v}_{0}$. For each of the $d_{k}$ children of
such a vertex, choose a leaf which is a descendent of that child (or the child itself if
it is a leaf). This results in a total of $q = d_{0}d_{1} \cdots d_{k} > 1$ distinct leaves
which we label $\mathbf{a}_{1}, \mathbf{a}_{2}, \ldots, \mathbf{a}_{q}$.
Set $\mathbf{b}_{j} = \mathbf{v}_{0}$ for all $j$ such that $1 \leq j \leq q$.
For all $i$ and $j$ we have $\rho_{\ell}(\mathbf{a}_{i}, \mathbf{b}_{j}) = M_{n}$.
Moreover, for all $i \not= j$, we have $\rho_{\ell}(\mathbf{a}_{i}, \mathbf{a}_{j}) \geq 2(M_{n} - M_{k})$
and $\rho_{\ell}(\mathbf{b}_{i}, \mathbf{b}_{j}) = 0$. It follows that any generalized
roundness exponent $p$ of $(T, \rho_{\ell})$ must satisfy
\begin{align}\label{inq:2}
\frac{1}{2}q(q-1)\left( 2(M_{n} - M_{k}) \right)^{p} \leq
\sum\limits_{i < j} \left\{ \rho_{\ell}(\mathbf{a}_{i}, \mathbf{a}_{j})^{p}
+ \rho_{\ell}(\mathbf{b}_{i}, \mathbf{b}_{j})^{p} \right\} \leq
\sum\limits_{i,j} \rho_{\ell}(\mathbf{a}_{i}, \mathbf{b}_{j})^{p} = q^{2}M_{n}^{p}.
\end{align}
By comparing the left and right sides of (\ref{inq:2}) it follows that $p$ must satisfy:
\begin{align*}
p \leq \frac{\ln\left(2+\frac{2}{(d_{0}d_{1}\cdots d_{k}) - 1}\right)}{\ln\left(2-\frac{2M_{k}}{M_{n}}\right)}.
\end{align*}
As $p$ was an arbitrary generalized roundness exponent of $(T, \rho_{\ell})$, we conclude that the lemma holds.
\end{proof}

\begin{theorem}\label{thm:4}
Let $(T, \mathbf{v}_{0})$ be a countable SST with downward degree sequence $(d_{0}, d_{1}, d_{2}, \ldots)$
and downward length sequence $\ell = (l_{0}, l_{1}, l_{2}, \ldots)$. If $d_{i} > 1$ for infinitely many $i$
and $\sum l_{i} = \infty$, then $\wp_{(T, \rho_{\ell})} = 1$.
\end{theorem}

\begin{proof}
For each positive integer $n$ let $(T_{n}, \mathbf{v}_{0})$ denote the finite SST with downward degree
sequence $(d_{0}, d_{1}, \ldots, d_{n-1})$ and downward length sequence $\ell = (l_{0}, l_{1}, \ldots, l_{n-1})$.
For each $k$, $1 \leq k \leq n$, set $M_{k} = \sum_{i=0}^{k-1} l_{i}$.
As $\sum l_{i} = \infty$ we will have $2l_{0} < l_{0} + l_{1} + \cdots + l_{n-1}$ provided $n$ is
sufficiently large. Moreover, for each such integer $n$, we may choose the largest integer $k = k(n)$
such that $M_{k} \leq \ln M_{n}$. As $n \rightarrow \infty$, the quantities $k, M_{k}, M_{n}$ and $d_{0}d_{1} \cdots d_{k}$
all tend to $\infty$. However, by construction, $(2M_{k})/M_{n} \rightarrow 0$ as $n \rightarrow \infty$.
Thus, by Lemma \ref{sst}, $\wp_{(T_{n}, \rho_{\ell})} \rightarrow 1$ as $n \rightarrow \infty$,
and so we conclude that $\wp_{(T, \rho_{\ell})} = 1$.
\end{proof}

\begin{theorem}\label{thm:5}
Let $f : \mathbb{N} \rightarrow (0, \infty)$ be an additively sub-exponential function.
Let $(T, \mathbf{v}_{0})$ be a countable SST with downward degree sequence $(d_{k})$.
Let $\ell$ denote the downward length sequence $(f(k))$.
If $d_{k} > 1$ for each integer $k \geq 0$, then $\wp_{(T, \rho_{\ell})} = 1$.
\end{theorem}

\begin{proof}
The condition $d_{k} > 1$ for each integer $k \geq 0$ ensures that the infinite comb $C(f)$ is isometric to a metric
subspace of $(T, \rho_{\ell})$. Thus $\wp_{(T, \rho_{\ell})} \leq \wp_{C(f)}$. Moreover, $\wp_{C(f)} = 1$ by
Theorem \ref{thm:3} and $\wp_{(T, \rho_{\ell})} \geq 1$, thereby forcing $\wp_{(T, \rho_{\ell})} = 1$.
\end{proof}

Theorem \ref{thm:5} provides examples of convergent SSTs with generalized roundness one.
For instance, we may simply set $d_{k} = 2$ and $f(k) = (k+1)^{-2}$ for all $k \geq 0$ to
obtain a countable SST that is convergent and has generalized roundness one. In particular,
such SSTs have finite diameter.

\section{Embedding properties of metric trees of generalized roundness one}\label{sec:5}

We conclude this paper with some comments on the special Euclidean embedding properties of metric
trees of generalized roundness one that set them apart from all other metric trees. As noted at the
outset of this paper, the notions of generalized roundness and negative type are equivalent. In
order to make this statement more precise we recall the following definition, the roots of which
can be traced back to an 1841 paper of Cayley \cite{Cay}.

\begin{definition}\label{types} Let $p \geq 0$ and let $(X,d)$ be a metric space. Then:
\begin{enumerate}
\item[(1)] $(X,d)$ has $p$-{\textit{negative type}} if and only if for all integers $k \geq 2$,
all finite subsets $\{x_{1}, \ldots , x_{k} \} \subseteq X$, and all choices of real numbers
$\eta_{1}, \ldots, \eta_{k}$ with $\eta_{1} + \cdots + \eta_{k} = 0$, we have
\begin{eqnarray}\label{TWO}
\sum\limits_{1 \leq i,j \leq k} d(x_{i},x_{j})^{p} \eta_{i} \eta_{j}  & \leq & 0.
\end{eqnarray}

\item[(2)] $(X,d)$ has \textit{strict} $p$-{\textit{negative type}} if and only if it has $p$-negative type
and the associated inequalities (\ref{TWO}) are all strict except in the trivial case
$(\eta_{1}, \ldots, \eta_{k})$ $= (0, \ldots, 0)$.
\end{enumerate}
\end{definition}

Lennard \textit{et al}.\ \cite{Ltw} proved that for all $p \geq 0$, a metric space $(X,d)$ has $p$-negative type
if and only if $p$ is a generalized roundness exponent of $(X,d)$. The significance of this result is that
builds a bridge between Enflo's \cite{En1} notion of generalized roundness and classical isometric embedding theory.
These connections, in conjunction with contemporary results on strict negative type, enable the following theorem.

\begin{theorem}\label{thm:6}
If $(T, d)$ is a metric tree of generalized roundness one, then:
\begin{enumerate}
\item[(1)] the metric transform $(T, \sqrt{d})$ is isometric to an affinely independent subset of $\ell_{2}$, and

\item[(2)] the metric transform $(T, \sqrt{d^{p}})$ does not embed isometrically into $\ell_{2}$ for any $p$, $1 < p \leq 2$.
\end{enumerate}
Moreover, the only metric trees that satisfy condition (2) are those of generalized roundness one.
\end{theorem}

\begin{proof}
The vertex set of $(T, d)$ is countable because it is a metric tree of generalized roundness one.
Our definitions imply that each finite subset of $T$ is contained in a finite subtree of $T$.
Hjorth \textit{et al}.\ \cite{Hj1} have shown that all finite metric trees have strict $1$-negative
type. Hence each finite metric subspace of $(T,d)$ has strict $1$-negative type. This ensures
that $(T,d)$ has strict $1$-negative type. Equivalently, the metric transform $(T, \sqrt{d})$
has strict $2$-negative type. Therefore $(T, \sqrt{d})$ is isometric to an affinely independent subset
of $\ell_{2}$ by Kelleher \textit{et al}. \cite[Theorem 5.6]{Ke2}. This establishes condition (1).

If the metric transform $(T, \sqrt{d^{p}})$ were to embed isometrically into $\ell_{2}$ for some $p \in (1, 2]$,
this would imply that $(T, d)$ has $p$-negative type. But by Lennard \textit{et al}.\ \cite{Ltw}, this
would mean that $p$ is a generalized roundness exponent of $(T, d)$, thereby contradicting
our assumption that $\wp_{(T, d)} = 1$. This establishes condition (1).

On the other hand, if a metric tree $(Z,d)$ is not of generalized roundness one, then it must be the
case that $\wp_{(Z,d)} > 1$ (because $1$ is a generalized roundness exponent of all metric trees).
By Lennard \textit{et al}.\ \cite{Ltw}, this implies that $(Z,d)$ has $p$-negative type for some $p \in (1, 2]$.
Consequently, the metric transform $(T, \sqrt{d^{p}})$ embeds isometrically into $\ell_{2}$ by
Kelleher \textit{et al}. \cite[Theorem 5.6]{Ke2}. We conclude that the only metric trees that satisfy
condition (2) are those of generalized roundness one.
\end{proof}

The proof of Theorem \ref{thm:6} shows that all metric trees have strict $1$-negative type. Therefore
every metric tree satisfies Theorem \ref{thm:6} (1) by the result of Kelleher \textit{et al}. \cite{Ke2}.

\section*{Acknowledgements}

The research in this paper was initiated at the 2011 Cornell University \textit{Summer Mathematics Institute}
(NSF grant DMS-0739338) and completed at the University of South Africa (Unisa).
The second named author was partially supported by the National Science Foundation
\textit{Graduate Research Fellowship Program} (NSF grant DGE-1144082). We are indebted to the
\textit{Visiting Researcher Support Programme} at Unisa, the US National Science Foundation, the
Department of Mathematics and the Center for Applied Mathematics at Cornell University for their support.

\bibliographystyle{amsalpha}

\begin{thebibliography}{Ar1}

\bibitem{Bir} A. Bird, G. Jameson, N. J. Laustsen,
              \textit{The Giesy-James theorem for general index $p$, with an application to operator ideals on the $p$th James space},
              J. Operator Theory \textbf{70} (2013), 291--307.

\bibitem{Cay} A. Cayley, \textit{On a theorem in the geometry of position}, Cambridge Mathematical Journal \textbf{II} (1841),
              267--271. (Also in \textsl{The Collected Mathematical Papers of Arthur Cayley (Vol.\ I)}, Cambridge University
              Press, Cambridge (1889), pp.\,1--4.)

\bibitem{Dah} A. M. Dahma and C. J. Lennard, \textit{Generalized roundness of the Schatten class, $C_{p}$},
              J. Math. Anal. Appl. \textbf{412} (2014), 676--684.

\bibitem{Djt} J. Diestel, H. Jarchow and A. M. Tonge, \textsl{Absolutely Summing Operators},
              Cambridge Studies in Advanced Mathematics \textbf{43} (1995), Cambridge U. Press, 1--474.

\bibitem{Din} S. Dineen, \textit{A Dvoretzky theorem for polynomials},
              Proc. Amer. Math. Soc. \textbf{123} (1995), 2817--2821.

\bibitem{Dcw} I. Doust, E. Caffarelli and A. Weston, \textit{Metric trees of generalized roundness one},
              Aeq. Math. \textbf{83} (2012), 239--256.

\bibitem{Dra} A.N. Dranishnikov, G. Gong, V. Lafforgue and G. Yu,
              \textit{Uniform embeddings into Hilbert space and a question of Gromov},
              Canad. Math. Bull. \textbf{45} (2002), 60--70.

\bibitem{En1} P. Enflo, \textit{On a problem of Smirnov}, Ark. Mat. \textbf{8} (1969), 107--109.

\bibitem{En2} P. Enflo, Uniform Structures and Square Roots in Topological Spaces II,
              Israel J. Math. \textbf{}8 (1970), 253--272.

\bibitem{Fav} T. Faver, K. Kochalski, M. Murugan, H. Verheggen, E. Wesson and A. Weston,
              \textit{Roundness properties of ultrametric spaces}, Glasgow Math. J. \textbf{56} (2014), 519--535.

\bibitem{Gie} D. P. Giesy and R. C. James, Uniformly non-$\ell_{1}$ and $B$-convex Banach spaces,
              Studia Math. \textbf{48} (1973), 61--69.

\bibitem{God} G. Godefroy, N. J. Kalton and G. Lancien, \textit{Szlenk Indices and uniform homeomorphisms},
              Trans. Amer. Math. Soc. \textbf{353} (2001), 3895--3918.

\bibitem{Haj} P. H\'{a}jek and M. Johanis, \textsl{Smooth Analysis in Banach Spaces},
              De Gruyter Series in Nonlinear Analysis and Applications \textbf{19} (2014), De Gruyter, 1--497.

\bibitem{Hj1} P. Hjorth, P. Lison\v{e}k, S. Markvorsen and C. Thomassen,
              \textit{Finite metric spaces of strictly negative type}, Linear Algebra Appl. \textbf{270} (1998), 255--273.

\bibitem{Ja0} R. C. James, \textit{Uniformly nonsquare Banach spaces}, Ann. of Math. \textbf{80} (1964), 542--550.

\bibitem{Ja1} R. C. James, \textit{Some self-dual properties of normed linear spaces}, Sympos. Infinite Dimensional Topology,
              Ann. of Math. Studies, no. 69, Princeton Univ. Press, Princeton, N. J., 1972.

\bibitem{Ja2} R. C. James, \textit{Super-reflexive Banach spaces}, Canad. J. Math. \textbf{24} (1972), 896--904.

\bibitem{Joh} W. B. Johnson, J. Lindenstrauss and G. Schechtman,
              \textit{Banach spaces determined by their linear structure}, GAFA \textbf{6} (1996), 430--470.

\bibitem{Ke1} C. Kelleher, D. Miller, T. Osborn and A. Weston,
              \textit{Strongly non-embeddable metric spaces}, Topology Appl. \textbf{159} (2011), 749--755.

\bibitem{Ke2} C. Kelleher, D. Miller, T. Osborn and A. Weston,
              \textit{Polygonal equalities and virtual degeneracy in $L_{p}$-spaces},
              J. Math. Anal. Appl. {\bf 415} (2014), 247--268.

\bibitem{Jbk} J. B. Kelly, \textit{Hypermetric spaces}, in: Proc. Conf. Geometry of Metric and Linear Spaces,
              Lecture Notes in Math. \textbf{490} (Springer-Verlag, Berlin, 1975), 17--31.

\bibitem{Kri} J. L. Krivine,
              \textit{Sous-espaces de dimension finie des espaces de Banach r\'{e}ticul\'{e}},
              Ann. of Math. \textbf{104} (1976), 1--29.

\bibitem{Laf} J-F. Lafont and S. Prassidis, \textit{Roundness properties of groups},
              Geom. Ded. \textbf{117} (2006), 137--160.

\bibitem{Lem} H. Lemberg, \textit{Nouvelle d\'{e}monstration d'un th\'{e}or\`{e}me de J. L. Krivine sur la finie
repr\'{e}sentation de $\ell_{p}$ dans un espace de Banach}, Israel J. Math \textbf{39} (1981), 341--348.

\bibitem{Ltw} C. J. Lennard, A. M. Tonge and A. Weston, \textit{Generalized roundness and negative type}, Michigan Math. J.
              \textbf{44} (1997), 37--45.

\bibitem{Lt2} C. J. Lennard, A. M. Tonge and A. Weston, \textit{Roundness and metric type},
              J. Math. Anal. Appl. \textbf{252} (2000), 980--988.

\bibitem{Hli} H. Li and A. Weston, \textit{Strict p-negative type of a metric space}, Positivity \textbf{14} (2010), 529--545.

\bibitem{Lin} J. Lindenstrauss and H. Rosenthal, \textit{The $\mathcal{L}_{p}$ spaces},
              Israel J. Math. \textbf{7} (1969), 325--349. 

\bibitem{Mau} B. Maurey and G. Pisier, \textit{S\'{e}ries de variables al\'{e}atoires vectorielles ind\'{e}pendantes
et propri\'{e}t\'{e}s g\'{e}om\'{e}triques des espaces de Banach}, Studia Math. \textbf{58} (1976), 45--90.

\bibitem{Epa} E. Prassidis and A. Weston,
              \textit{Manifestations of non linear roundness in analysis, discrete geometry and topology}.
              In: Arzhantseva, G., Valette, A. (eds.) Limits of Graphs in Group Theory and Computer Science.
              Research Proceedings of the \'{E}cole Polytechnique F\'{e}d\'{e}rale de Lausanne.
              CRC Press, Boca Raton (2009).

\bibitem{Rib} M. Ribe, \textit{On uniformly homeomorphic normed spaces}, Ark. Mat. \textbf{14} (1976), 237--244.
             
\bibitem{Ros} H. P. Rosenthal,
              \textit{On a theorem of J. L. Krivine concerning block finite representability of $\ell_{p}$ in general Banach spaces},
              J. Funct. Anal. \textbf{28} (1978), 197--225.

\bibitem{San} S. S{\'a}nchez, \textit{On the supremal $p$-negative type of a finite metric space},
              J. Math. Anal. Appl. \textbf{389} (2012), 98--107.

\bibitem{Sch} I. J. Schoenberg,
             \textit{On certain metric spaces arising from euclidean spaces by a change of metric and their imbedding in Hilbert space},
             Ann. of Math. \textbf{38} (1937), 787--793.

\bibitem{Ste} J. Stern, \textit{Some applications of model theory in Banach space theory},
              Ann. Math. Logic \textbf{9} (1976), 49--122.

\bibitem{Wes} A. Weston, \textit{On the generalized roundness of finite metric spaces}, J. Math. Anal. Appl.
              \textbf{192} (1995), 323--334.

\end{thebibliography}

\end{document}